\documentclass[a4paper,11pt]{article}
\title{A McShane-type identity for closed surfaces}
\author{\textsc{Yi Huang}\\
\textit{Department of}\\ 
\textit{Mathematics and Statistics}\\ 
\textit{University of Melbourne}\\
\href{mailto:huay@ms.unimelb.edu.au}{\texttt{huay@ms.unimelb.edu.au}}
}

\date{}

\usepackage{amsmath, amsfonts, amssymb, amsthm, amscd, graphicx, float, epstopdf, hyperref
}

\usepackage[small
]{titlesec}

\usepackage{eulervm}

\newtheorem{thm}{Theorem}
\newtheorem{por}[thm]{Porism}
\theoremstyle{definition}

\theoremstyle{remark}

\theoremstyle{plain}
\newtheorem{lem}[thm]{Lemma}
\newtheorem{cor}[thm]{Corollary}

\begin{document}
\maketitle
\parindent = 0cm
\small

\begin{abstract}
\noindent
We prove a McShane-type identity--- a series, expressed in terms of geodesic lengths, that sums to $2\pi$ for any closed hyperbolic surface with one distinguished point. To do so, we prove a generalized Birman-Series theorem showing that the set of complete geodesics on a hyperbolic surface with large cone-angles is sparse.
\end{abstract}

\section*{Introduction}
\pagestyle{empty}
In his PhD thesis \cite{mcshane_thesis}, McShane obtained the following beautiful identity summed over the collection $\mathcal{C}$ of all simple closed geodesics on a hyperbolic one cusped torus $S_{1,1}$:
\begin{align*}
\sum_{\gamma\in\mathcal{C}} \frac{1}{1+\exp{\ell\gamma}}=\frac{1}{2},
\end{align*}
where $\ell{\gamma}$ denotes the hyperbolic length of the closed geodesic $\gamma$. By doubling both sides, these summands may be interpreted in terms of probability: cutting $S_{1,1}$ along any simple closed geodesic $\gamma\in\mathcal{C}$ gives us a pair of pants $P_{\gamma}$, and $\frac{2}{1+\exp{\ell\gamma}}$ is the chance that a geodesic shooting out from the cusp in $S_{1,1}$ hits itself before leaving $P_{\gamma}$. These probabilities sum to $1$ since the Birman-Series theorem \cite{birman_series} on the sparsity of complete geodesics informs us that almost all geodesics are self-intersecting. In subsequent papers \cite{mcshane_allcusps, mcshane_weir}, this work was extended to include identities at the Weierstrass point of the 1-cusped torus 
and for surfaces with more cusps and genera. 
\newline
\newline
Mirzakhani then generalized these identities to hyperbolic surfaces with geodesic boundary \cite{mirz_simp}, and used them to unfold the volume of a moduli space of bordered Riemann surfaces over topologically simpler moduli spaces. In so doing, she obtained explicit recursions for the Weil-Petersson volumes of these moduli spaces. She then used symplectic reduction to interpret these volumes as intersection numbers \cite{mirz_proof}, and derived a new proof for (most of) Witten's conjecture \cite{witten_conjecture}. 
\newline
\newline
Almost concurrently, Tan-Wong-Zhang \cite{tan_zhang_cone} independently derived Mirzakhani's generalized McShane identities for bordered surfaces. They also obtained such identities for hyperbolic surfaces with cone points with angles $\leq\pi$, extending the Birman-Series sparsity theorem out of necessity. In particular, they observed that the cone point identities are analytic continuations of the geodesic boundary identities, where real boundary lengths are replaced with imaginary ones. Their cone-angle constraint is because certain pairs of pants fail to exist for angles $>\pi$.
\newline
\newline
Until fairly recently, the only McShane identity known for a closed hyperbolic surface was the following result taken from \cite{mcshane_genus2}: let $\mathcal{A}$ be the set of all pairs $(\gamma_1,\gamma_2)$ of disjoint simple closed geodesics on a genus-2 surface $S_2$ so that $\gamma_1$ is separating and $\gamma_2$ is non-separating, then:
\begin{align}
\sum_{(\gamma_1,\gamma_2)\in\mathcal{A}} \mathrm{arctan}\exp\left(-\frac{\ell\gamma_1}{4}-\frac{\ell\gamma_2}{2}\right)=\frac{3\pi}{2}\label{genus2},
\end{align}

Like their cusped case cousins, these sums may be interpreted by classifying geodesic arcs emanating from the images of the six Weierstrass points in the quotient surface of $S_{2}$ under its hyperelliptic involution. During a 2010 conference talk \cite{tan_talk} at the National University of Singapore, Tan outlined how one might obtain such an identities for closed surfaces with one marked point by shooting out in opposite directions at the same speed from a fixed point. Although the relevant summands for such an identity are hard to obtain \cite{tan_chat}, Luo and Tan have computed the integral over the entire surface as one varies the marked point \cite{luo_tan}. The resulting identity is expressed in terms of dilogarithms --- much like Bridgeman's orthospectrum identity \cite{orthospectra}.
\newline
\newline
In this paper, we deduce an identity for closed hyperbolic surfaces $S$ equipped with a marked point $p$ of the form:
\begin{align*}
\sum_{P\in\mathcal{HP}(S,p)}\mathrm{Gap}(P)=2\pi,
\end{align*}
where the function $\mathrm{Gap}$ depends on the geometry of immersed half-pants on $S$. The main difference between this identity and its predecessors is that we categorize geodesics emanating from $p$ by the lasso-induced geodesic half-pants in which they lie. 
This is all made possible by first extending the Birman-Series theorem.
\newline
\newline
Please note that we often implicitly invoke existence proofs for unique geodesic representatives of a essential homotopy class on a hyperbolic surface. One version that suits our purposes may be found in Lemma 7.3 of \cite{mondello_cone}, which essentially states that even for hyperbolic surfaces with large cone-angles, such a geodesic representative still exists and is unique, but may be broken at cone points with angle $\geq\pi$. Moreover, although we invoke the Gauss-Bonnet theorem at times, in practice all we need to know is that geodesic monogons and bigons do not exist in the hyperbolic world, and that the area of a geodesic triangle is equal to $\pi$ minus its three internal angles.

\section{Birman-Series theorem}
\pagestyle{plain}

Our extended Birman-Series theorem on the sparsity of geodesics may be stated as follows:
\begin{thm}\label{bf_extended}
Given any complete finite-volume hyperbolic surface $S$ with a finite collection $C$ of cone points, fix an integer $k$. Then the points constituting all complete hyperbolic geodesics possibly broken at $C$ with at most $k$ intersections is nowhere dense on $S$ and has Hausdorff dimension 1.
\end{thm}
Note that an immediate corollary of the collection of complete simple geodesics on $S$ being Hausdorff dimension $1$ is that it has zero Lebesgue measure.
\newline
\newline
Our extended result differs from previous ones \cite{birman_series, tan_zhang_cone} in that we allow for cone-angles and broken geodesics. In particular, the fact that a cone point with angle $2\pi$ is equivalent to a marked point gives us the following corollary:
\begin{cor}
Given a complete finite-volumed hyperbolic surface $S$ and a countable collection of points $C\subset S$, the set of points which lie on geodesics broken at finitely many points in $C$ with finitely many self-intersections has zero Lebesgue measure.
\end{cor}
For our purposes, we require the following corollary of the Birman-Series theorem:
\begin{cor}
Consider the usual length $2\pi$ Borel measure on the unit tangent space of $p\in S$--- thought of as the space of directions from $p$. Almost every direction from $p$ projects to a geodesic that's self-intersecting.
\end{cor}


Our proof of the Birman-Series theorem is organized as follows:
\begin{enumerate}
\item
Take a geodesic polygonal fundamental domain $R$ with the restricted covering map $\pi:R\rightarrow S$. Show that the number of isotopy classes of $n$-segmented geodesic arcs on $S$ with endpoints on $\pi(\partial R)$ grows polynomially in $n$. 
\item
Show that, with respect to $n$, an exponentially decreasing neighborhood of any representative of such an isotopy class will cover all other representatives of the same class.
\item
By increasing $n$, we prove that the area covered by such geodesic arcs is bounded by a polynomial divided by an
 exponential and must tend to $0$, and use this to obtain the desired results.
\end{enumerate}

\subsection{Notation and proof}
Most of our proof is taken from the original Birman-Series paper, with minor modifications in presentation. Throughout the proof, we assume that intersections are counted with multiplicities. This doesn't affect the actual result because a complete non-closed geodesic with infinitely many self-intersections at only finitely many points cannot exist.
\newline
\newline
For our intents and purposes, a hyperbolic surface $S$ with a finite set of \emph{cone-points} $C$ may be thought of as a topological surface  equipped with a smooth constant curvature $-1$ Riemannian metric on the open set $S-C$, such that the local geometry of a neighborhood around each $p\in C$ may be modeled by taking an angle $\theta_p$ wedge in the hyperbolic plane and radially identifying the two straight edges of this wedge. We call $\theta_p$ the \emph{cone-angle} at $p$. A precise definition may be found in \cite{cone_metric} and \cite{cone_metric2}.\newline
\newline
Let $S$ be a hyperbolic surface with cone-points $C$, and let $R$ be a hyperbolic polygonal fundamental domain for $S$ such that the covering map $\pi:R\rightarrow S$ surjects the vertices of $R$ onto $C$. That such a polygon exists is a simple exercise in applying lemma 7.3 of \cite{mondello_cone}, bearing in mind that $R$ might not embed in the hyperbolic plane.
\newline
\newline
Let $J_k$ denote the set of geodesic arcs with at most $k$ self-intersections on $S$ which start and end on $\pi(\partial R)$, and let $J_k(n)$ be the subset of geodesics $\gamma$ in $J_k$ cut up into $n$ intervals by $\pi(\partial R)$. Moreover, let $[J_k]$ denote the equivalence classes of $J_k$ with respect to isotopies that leave invariant each connected component of $\pi(\partial R)\setminus C$. Define $[J_k(n)]$ similarly.
\newline
\newline
Finally, given $[\gamma]\in [J_k]$ representing some geodesic arc $\gamma\in J_k$, the restriction to $R$ of all lifts of $\gamma$ to the universal cover of $S$ constitutes a set of geodesic segments. Each segment corresponds to an element of $J_k(1)$,\footnote{The $k$ is not used here.} and we call the ordered set of isotopy arc classes obtained in this way a \emph{diagram} on $R$. In particular, we call elements of $[J_0]$ \emph{simple diagrams}. Informed readers may notice that this definition of a diagram differs slightly from those found in \cite{birman_series} and \cite{tan_zhang_cone}.

\begin{figure}[H]
\begin{center}
\includegraphics[scale=0.375]{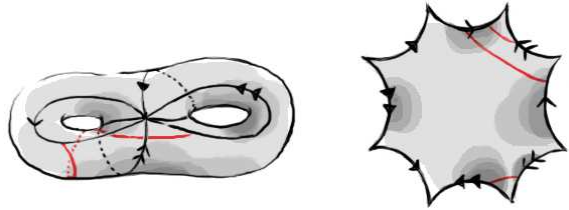}
\caption{An example of a simple diagram (in red).}
\label{diagram}
\end{center}
\end{figure}

The following two lemmas are slight generalizations of lemmas $2.5$ and $3.1$ of \cite{birman_series}, and of lemmas $8.2$ and $8.3$ of \cite{tan_zhang_cone}. In lemma \ref{polynomial}, by tweaking the definition of diagrams to allow segments running between the vertices of $\pi(\partial R)$, we enable our results to extend to broken geodesics. And in lemma \ref{linear}, we introduce Gauss-Bonnet based arguments to cover a concern that arises when a cone-angle is greater than $\pi$.

\begin{lem}
The number of elements in $[J_k(n)]$ is bounded above by a polynomial in $n$.\label{polynomial}
\end{lem}

\begin{proof}
By construction, a diagram identifies under the covering map to give an element of $[J_k]$. Therefore, the cardinality of $[J_k(n)]$ is bounded above by the number of types of diagrams comprised of $n$ segments. Simple diagrams may be specified by saying which segment we start and which we end on, as well as how many of each type of segment there is. Let $m$ denote the number of sides of $R$, since there are ${2m\choose2}$ types of segments (with respect to the appropriate isotopies):
\begin{align}
\mathrm{Card}[J_0(n)]\leq {n^{2}}{{2m\choose2}+n-1\choose n-1}=:P(n).
\end{align}
For non-simple diagrams, merely specifying the starting, ending segments and the number of each type of segment is insufficient to recover the diagram because there is a degree of freedom for how these segments intersect. In particular, having specified how many of each type of segment there is, if we arbitrarily label these segments from $1$ to $n$, then knowing whether two segments intersect is sufficient to recover the data of the whole diagram. Since two segments may intersect at most once, we see that the degree of freedom introduced by this flexibility in intersection is bounded above by the number of ways of picking $k$ intersections out of all the types of possible intersections: ${n\choose 2}$. Therefore:
\begin{align}
\mathrm{Card}[J_k(n)]\leq P(n){{n\choose 2}\choose k},
\end{align}
and hence the number of elements in $[J_k(n)]$ is bounded by a polynomial in $n$.
\end{proof}

\begin{lem}
The length of a geodesic arc $\gamma\in J_k(n)$ grows at least linearly in $n$ for $n$ sufficiently large. That is: 
$$\ell\gamma\geq\alpha n.$$ \label{linear}
\end{lem}

\begin{proof}
Let $m$ denote the number of sides of $R$, and let $\bar{\gamma}$ be the geodesic arc given by $m(2k+1)$ consecutive segments of a geodesic $\gamma\in J_k$. It suffices to show that one of these $m(2k+1)$ segments of $\bar{\gamma}$ is of length $\geq\rho$, because we can take $\alpha=\frac{\rho}{2m(2k+1)}$. Assign $\rho$ as the length of the shortest geodesic joining two non-adjacent edges of $R$. Since hyperbolic monogons and bigons do not exist, the only way that one might have a segment of length less than $\rho$ in $R$ is to travel between two adjacent edges. Assume that all $m(2k+1)$ segments of $\bar{\gamma}$ are projections of arcs which join adjacent edges in $R$. Hence $\bar{\gamma}$ spirals at least $2k+1$ times around a cone point $c$, and intersects some edge $e\in\pi(\partial R)$ at least $2k+1$ times. Cutting along $e$ and $\bar{\gamma}$ yields a hyperbolic triangle with internal angle $\theta$ at $c$, and Gauss-Bonnet tells us that $\theta<\pi$. But this is precisely what has previously been covered by Birman-Series ($\theta=0$) and Tan-Wong-Zhang ($0<\theta<\pi$), who showed that this would result in the existence of $k+1$ intersections on $\gamma$ - thereby giving us the desired contradiction.
\end{proof}

\begin{lem}\label{exponential}
Given $\gamma_1,\gamma_2\in J_k(2n+1)$ representing the same isotopy class in $[J_k(2n+1)]$, let $\delta_1$ and $\delta_2$ denote their respective middle (i.e., $n$th) segments. Then, for $n$ large enough, $\delta_1$ lies within a $ce^{-\alpha n}$ neighborhood of $\delta_2$.
\end{lem}

\begin{proof}
Since $\gamma_1$ and $\gamma_2$ are homotopic and unbroken geodesics, they have at most one intersection and the strip between them will not contain any cone-points. Hence, we can develop this thin region locally in $\mathbb{H}$. Cutting this strip at the shared starting point/edge of the $\gamma_i$, we obtain a (potentially self-intersecting) geodesic polygon with two long sides given by the $\gamma_i$ and the remaining two short sides with length bounded by the longest edge in $R$. Since this setup is now independent of cone points, the analogous result --- lemma $3.2$ of the original Birman-Series paper gives us the desired computation that an $\alpha n$-long hyperbolic polygon will have a $ce^{-\alpha n}$-thin waist region.
\end{proof}

Finally, we prove theorem \ref{bf_extended}. Our arguments are taken from section $4$ of \cite{birman_series}. 

\begin{proof}
First select one geodesic representative for each class in $[J_k(2n+1)]$ and denote this collection of $n$-th segments of these geodesics by $F_n$. Let $S_k$ denote the collection of points which lie on a complete hyperbolic geodesic on $S$ with at most $k$ intersections. Any $x\in S_k$ lies on some arc $\gamma\in J_k$ and hence on the middle segment of the corresponding $(2n-1)$-segmented subarc of $\gamma$ denoted by $\bar{\gamma}\in J_k(2n-1)$. But, by lemma \ref{exponential}, we know that $x$ must lie in a $ce^{-\alpha n}$ neighborhood of any representative of $[\bar{\gamma}]$ in $[J_k(2n-1)]$. Therefore, the $ce^{-\alpha n}$ neighborhood of $F_n$ covers $S_k$. Since the cardinality of $F_n$ is bounded by a polynomial $P_k(n)$ and the length of each segment in $F_n$ is at most $\mathrm{diam}(R)$, the closure of the $ce^{-\alpha n}$-neighborhood of $F_n$ has measure bounded by $c'e^{-\alpha n}P(n)$, where $c'$ is determined by $c$ and $\mathrm{diam}(R)$. Thus, $S_k$ lies in the intersection of a collection of closed sets with arbitrarily small measure, and must be nowhere dense. 
\newline
\newline
To obtain that $S_k$ is Hausdorff dimension $1$, we first note that its dimension must be at least $1$ because it contains geodesics. On the other hand, we can cover the $ce^{-\alpha n}$ neighborhood of $F_n$ with $\lceil\mathrm{diam}(R)/(2ce^{-\alpha n})\rceil$ balls of radius $2ce^{-\alpha n}$ to show that the Hausdorff content $C^d_H(S_k)=0$ for $d=1$. Therefore, the Hausdorff dimension of $S_k$ is bounded above by $1$.
\end{proof}

\section{Closed Surface Identity}
\pagestyle{plain}
In this section, we first introduce the notion of lasso-induced hyperbolic half-pants. Several small lemmas then lead to the proof of our main theorem.


\subsection{The geometry of half-pants}

A thrice punctured sphere endowed with a hyperbolic metric so that each of its boundary components is either a cone-point, cusp or a closed geodesic is called a hyperbolic \emph{pair of pants}. For each boundary component $\beta$, there is a unique shortest geodesic arc starting and ending at $\beta$ which cannot be homotoped into the boundary. We call this geodesic the \emph{zipper} with respect to $\beta$.
\newline
\newline
We call any connected components obtained by cutting a pair of pants along one of its three zippers a pair of hyperbolic \emph{half-pants}, and we call its non-zipper closed boundary component its \emph{cuff}. Any pair of half-pants may be obtained by gluing together two orientation-reversing-isometric hyperbolic polygons as in figure \ref{half-pants-halved}. Gluing this isometry with its inverse induces an orientation reversing involution on a pair of half-pants.

\begin{figure}[H]
\begin{center}
\includegraphics[scale=0.375]{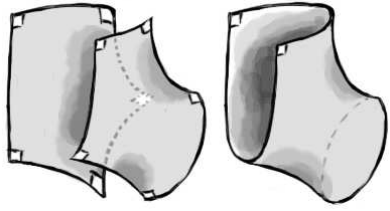}
\caption{A dissected pair of half-pants}
\label{half-pants-halved}
\end{center}
\end{figure}

Henceforth, all half-pants $P$ relevant to our purposes result from cutting a pair of pants, with one cone-pointed boundary $C$ and two geodesic boundaries, along the zipper $\zeta$ with respect to $C$. Given any pair of half-pants $P$, there are precisely two simple complete geodesic rays $r_1, r_2$ starting at $C$ and spiraling arbitrarily close to the cuff of our half-pants. Due to the reflection isometry of $P$, the angle at $C$ between $\zeta$ and $r_1$ is the same as that between $\zeta$ and $r_2$; we call this the \emph{spiral-angle} of $P$ at $C$. 

\begin{figure}[h]
\begin{center}
\includegraphics[scale=0.375]{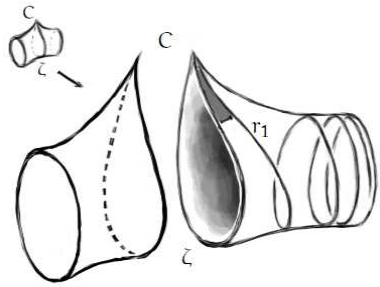}
\caption{A shaded spiral-angle region}
\end{center}
\end{figure}
\begin{lem}\label{spiral-angle}
Given a pair of hyperbolic half-pants $P$ with (ex-)cone-point boundary $C$, zipper $\zeta$ and cuff $\gamma$, its spiral-angle is:
\begin{align}
\mathrm{arcsin}\left(\frac{\cosh(\frac{\ell\gamma}{2})}{\cosh(\frac{\ell\zeta}{2})}\right)-\mathrm{arcsin}\left(\frac{\sinh(\frac{\ell\gamma}{2})}{\sinh(\frac{\ell\zeta}{2})}\right).\label{bangle}
\end{align}
\end{lem}

\begin{proof}
By cutting $P$ along the shortest geodesic from $C$ to $\gamma$ and then the shortest geodesic between $\zeta$ and $\gamma$, we decompose $P$ into two isometric quadrilaterals with three right angles. For either quadrilateral, let its acute angle be of magnitude $\theta_1$, then:
\begin{align}
\sin(\theta_1)\cosh(\frac{\ell\zeta}{2})=\cosh(\frac{\ell\gamma}{2}).\label{(2)}
\end{align}
See, for example, Theorem $2.3.1$ of \cite{buser}. By considering one lift of this quadrilateral to the universal cover $\tilde{P}\subset\mathbb{H}$ of the half-pants $P$ with an appropriate lift of $r_i$, as shown in figure \ref{universal-cover}, we see that the gap angle is given by $\theta_1$-$\theta_2$. For the triangle bounded by the depicted lifts of $r_i$, $\gamma$ and the shortest geodesic from $C$ to $\gamma$,
\begin{align}
\sin(\theta_2)\cosh(\mathrm{d}_{\mathbb{H}(C,\gamma)})&=1.\label{(3)}
\end{align}
And for the quadrilateral comprising half of the half-pants $P$:
\begin{align}
\cosh(\mathrm{d}_{\mathbb{H}(C,\gamma)})\sinh(\frac{\ell\gamma}{2})&=\sinh(\frac{\ell\zeta}{2}).\label{(4)}
\end{align}
See, for example, Theorem $2.3.1$ and Theorem $2.2.2$ of \cite{buser}. Putting $\eqref{(2)}$, $\eqref{(3)}$ and $\eqref{(4)}$ together, we obtain the magnitude of the cone-angle.
\end{proof}

\begin{figure}[h]
\begin{center}
\includegraphics[scale=0.42]{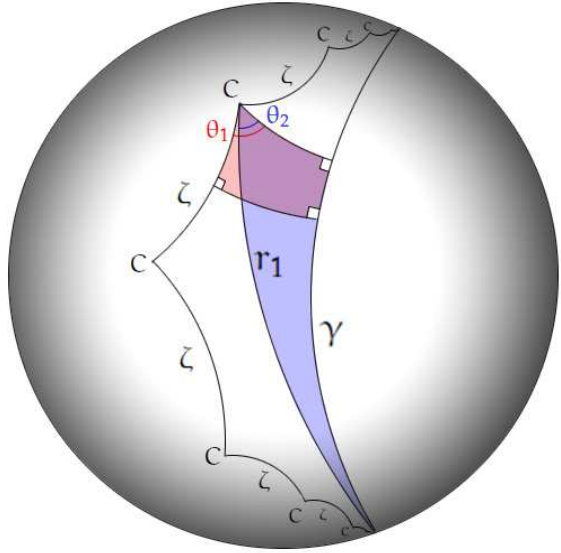}
\caption{The universal cover of $P$}
\label{universal-cover}
\end{center}
\end{figure}

\begin{lem}\label{lies-within}
Given a pair of half-pants $P$ embedded in some surface $S$ labeled as above, the segment of any geodesic ray projecting from $C$ up to its first self-intersection lies completely in $P$ iff it projects within one of the two spiral-angles of $\zeta$.
\end{lem}

\begin{proof}
We see from the universal cover of $P$ that any geodesic $\alpha$ launched within a spiral-angle of $\zeta$ must intersect $\zeta$ since it cannot intersect $r_i$ without a bigon forming. Any lift of the segment $\bar{\alpha}$ of $\alpha$ up to its first intersection with $\zeta$ will, along with the lifts of $\zeta$, bound a polygonal region in $\tilde{P}$. 
 Then consider a different lift of $\bar{\alpha}$ starting from within this polygon, by the minimality of the intersection time of $\bar{\alpha}$ with $\zeta$, this different lift cannot intersect the sides of our polygon corresponding to $\zeta$. Thus, we see that $\alpha$ intersects itself before leaving $P$ via $\zeta$, giving us the \emph{if} part of the claim. 
 \newline
 \newline
 On the other hand, any geodesic $\beta$ launched outside of either spiral-angle regions will first intersect $\gamma$ at some angle $\phi>0$. If two lifts of $\beta$, denoted by ${\beta}_1$ and ${\beta}_2$, intersect before leaving $\tilde{P}$, then $\beta_1$, $\beta_2$ and $\gamma$ bound a hyperbolic triangle. Gauss-Bonnet then tells us that the angles, measured clockwise, between the $\beta_i$ and $\gamma$ differ by at least $\phi$. We now have two distinct values for the angle at which $\beta$ first intersects $\gamma$; this contradiction gives us the desired converse.
\end{proof}

\subsection{Lassos and lasso-induced hyperbolic half-pants}

Let $\alpha$ be a geodesic ray shooting out from $p$. By the Birman-Series theorem, $\alpha$ is almost always self-intersecting, and we call the geodesic segment of $\alpha$, up to its first self-intersection the \emph{lasso} of $\alpha$. Further, we call the segment from $p$ to the intersection point of a lasso its \emph{spoke}, and the simple closed geodesic broken at the intersection point of a lasso its \emph{loop}. A lasso emanating from $p$ naturally determines an immersed pair of half-pants on $S$ as follows:

\begin{lem}\label{lasso}
The lasso of any non-simple geodesic ray emanating from $p$ induces and is contained in an isometrically immersed pair of half-pants.
\end{lem}

\begin{proof}
 The closed path obtained by traversing the length of the lasso of $\alpha$ and back along its spoke is a representative of a simple element of $\pi_1(S,p)$. Let $\gamma_p$ be its unique geodesic representative on the surface $S$ with the condition that $\gamma_p$ must begin and end at $p$ - that such a curve exists may be obtained from curve-shortening arguments, or by considering the universal cover of $S$. Note that it \emph{is} possible for $\gamma_p$ to be self-intersecting, although it is locally geodesic everywhere except at its end-points.\newline
 \newline
 Let $\gamma$ denote the unique simple geodesic representative of $[\gamma_p]\in\pi_1(S)$ on the closed surface $S$. These two geodesics $\gamma_p$ and $\gamma$ are homotopic in $S$ and bound the immersed image $\iota(P)$ of a pair of half-pants $P$. 
\newline
\newline
We need now to show that our lasso is contained in $\iota(P)$. This is a natural consequence of an adapted version of a curve-shortening procedure attributed to Semmler (Appendix in \cite{buser}), but can also be seen as follows: consider the universal cover of the loop of our lasso developed in $\mathbb{H}$, and add in all lifts of the spoke of our lasso adjoining this infinite quasi-geodesic. The convex (hyperbolic) hull of this shape in $\mathbb{H}$ is a universal cover for $P$, therefore our lasso must lie within $P$ and hence within $\iota(P)$.

\end{proof}

We say that an immersed pair of half-pants is \emph{lasso-induced at $p$} if it is induced by the lasso of a geodesic ray emanating from $p$. Topologically, there are three types of lasso-induced half-pants $\iota(P)$:
\begin{enumerate}
\item 
if $\gamma_p$ is simple and does not intersect $\gamma$, then $\iota(P)$ is embedded;
\item
if $\gamma_p$ self-intersects, but does not intersect $\gamma$, then $\iota(P)$ is a thrice-holed sphere;
\item 
if $\gamma_p$ is simple, but does intersect $\gamma$, then $\iota$ is a torus with a hole.
\end{enumerate}

We call the respective images of the cuff or zipper of $P$ under the isometric immersion $\iota$ the cuff or zipper of $\iota(P)$. The case where $\gamma_p$ is not simple and intersects $\gamma$ does not arise because the spoke of the lasso is then forced to  intersect its loop at least twice (counted with multiplicity), thereby contradicting the definition of a lasso.
\begin{figure}[h]
\begin{center}
\includegraphics[scale=0.75]{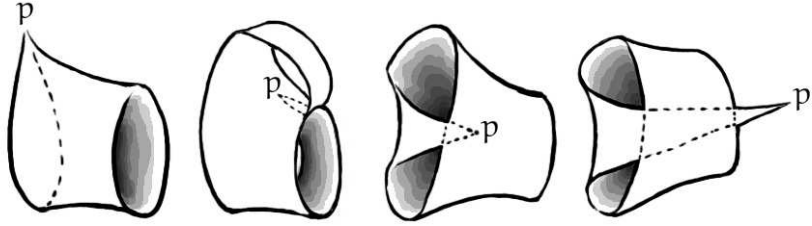}
\caption{The left-most three figures are examples of lasso-induced half-pants, whereas the last is not.}
\label{prototypes}
\end{center}
\end{figure}


\subsection{The main theorem}

There are $2\pi$ worth of directions emanating from $p$, and the extended Birman-Series theorem tells us that almost all geodesics shot out from $p$ will be self-intersecting. Lemma \ref{lasso} tells us that the set of all self-intersecting geodesic rays emanating from $p$ may be partitioned based on which lasso-induced pair of half-pants it induces, and this produces a partition of $2\pi$.
\newline
\newline
Lemma \ref{lies-within} then says that for each \emph{embedded} pair of half-pants $P$, there is
\begin{align*}
2\mathrm{arcsin}\left(\frac{\cosh(\frac{\ell\gamma}{2})}{\cosh(\frac{\ell\gamma_p}{2})}\right)-
2\mathrm{arcsin}\left(\frac{\sinh(\frac{\ell\gamma}{2})}{\sinh(\frac{\ell\gamma_p}{2})}\right)
\end{align*}
worth of directions from $p$ that shoot out geodesics whose lassos lie in $P$. Since $P$ is topologically an annulus, each such lasso must induce $P$.
\newline
\newline
In the case that a lasso-induced pair of half-pants $P$ is not embedded, it is possible for a lasso to lie within $P$ but not to be launched within one of its spiral-angle regions. Fortunately, such lassos cannot induce $P$ because their loops have non-zero algebraic intersection with the cuff of $P$ - which is necessarily homotopic to the loop of any lasso that induces $P$. Therefore, the spiral-angle does not under-count the lassos which induce $P$. However, it is also possible for a lasso contained in $P$ to not induce $P$. This means that the spiral-angle over-counts the lassos which induce $P$, and we need only to compute and subtract the angle of the regions corresponding to these non-inducing lassos to produce a McShane-type identity.

\begin{thm}
Given a closed hyperbolic surface $S$ with marked point $p$, let $\mathcal{HP}(S,p)$ denote the collection of half-pants lasso-induced at $p$. We define the real function $\mathrm{Gap}:\mathcal{HP}(S,p)\rightarrow[0,\pi]$ to output the \emph{gap-angle} of the directions from $p$ that shoot out geodesics whose lassos lie in $P$. Then,
 
\begin{align}
\sum_{P\in\mathcal{HP}(S,p)}\mathrm{Gap}(P)=2\pi.\label{mainthm}
\end{align}
\end{thm}

We close this section by defining and proving the $\mathrm{Gap}$ function in terms of explicit length parameters on the input pairs of lasso-induced half-pants $P$.
\newline
\newline
For embedded pairs of half-pants $P\in\mathcal{HP}(S,p)$, we know from previous discussions that the gap-angle is double the spiral-angle, as given by \eqref{bangle}.
\newline
\newline
When $P$ is topologically a thrice-holed torus, we need two pieces of geometric information from $P$ to define its gap-angle. First we must know the position of $p\in P$, which we specify using two parameters $\tau$ and $\delta$: we know that $P$ is the isometric immersion of a unique pair of half-pants $\tilde{P}$. There are two preimages for $p$ in $\tilde{P}$ and there is a unique way to reach the preimage of $p$ on the interior of $\tilde{P}$ by launching orthogonally from the cuff of $\tilde{P}$ as per the black dotted line in the left figure in figure \ref{definedt}. We set $\tau\in[0,\ell\gamma)$ to parametrize the position of the launching point on the cuff, so that the point on the cuff which orthogonally projects to the tip of the zipper is set to $0$; the parameter $\delta$ then denotes the distance between the interior preimage of $p$ and the cuff of $\tilde{P}$.

\begin{figure}[h]
\begin{center}
\includegraphics[scale=0.60]{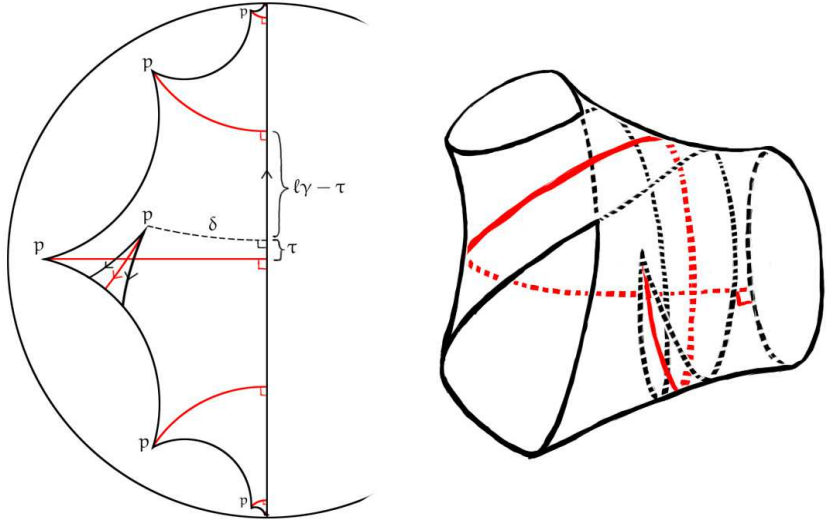}
\caption{The parameter $n$ is $-1$ in this case because it goes against the orientation on the cuff in which $\tau$ is increasing.}
\label{definedt}
\end{center}
\end{figure}

The second piece of information we require counts (with sign) how many times the tip of the zipper of $P$ wraps around itself. Specifically, consider the unique shortest geodesic $\beta$ between the boundary/zipper preimage of $p$ in $\tilde{P}$ and the cuff of $\tilde{P}$ (as shown in red). We define $n$ to be the number of times $\iota(\beta)$ intersects itself, signed to be positive if $\beta$ shoots out from $p$ in the same direction that $\tau$ is increasing, and negative in the direction that $\tau$ is decreasing. We refer to figure \ref{definedt}  and its caption for an example. Note that specifying these parameters does not specify the whole geometry of $P$. 
\newline
\newline
Given this setup, if $P$ is topologically a thrice-holed sphere and $n=0$, then the gap-angle of $P$ is:
\begin{align}
&\mathrm{Gap}(P)=\mathrm{Gap}(\ell\gamma,\ell\gamma_p,\tau,\delta,n=0)\notag\\
&=\mathrm{max}\left\{\Theta(\delta,\tau,\mathrm{arccosh}\left(\frac{\sinh(\frac{\ell\gamma}{2})}{\sinh(\frac{\ell\gamma_p}{2})}\right))-\mathrm{arcsin}\left(\frac{\sinh(\frac{\ell\gamma}{2})}{\sinh(\frac{\ell\gamma_p}{2})}\right),0\right\}\notag \\
&+\mathrm{max}\left\{\Theta(\delta,\ell\gamma-\tau,\mathrm{arccosh}\left(\frac{\sinh(\frac{\ell\gamma}{2})}{\sinh(\frac{\ell\gamma_p}{2})}\right))-\mathrm{arcsin}\left(\frac{\sinh(\frac{\ell\gamma}{2})}{\sinh(\frac{\ell\gamma_p}{2})}\right),0\right\},
\end{align}
where $\Theta(x,y,z)$ is defined by:
\begin{align*}
\Theta(x,y,z)=\frac{1}{2}\mathrm{arccos}\left(\frac{2(\cosh(x)\cosh(y)\sinh(z)-\sinh(x)\cosh(z))^2}{(\cosh(x)\cosh(y)\cosh(z)-\sinh(x)\sinh(z))^2-1)}-1\right).
\end{align*}
And if $n\neq0$, then the gap-angle of $P$ is:
\begin{align}
&\mathrm{Gap}(P)=\mathrm{Gap}(\ell\gamma,\ell\gamma_p,\tau,\delta,n)=
\Theta(\delta,|n\ell\gamma-\tau|,\mathrm{arccosh}\left(\frac{\sinh(\frac{\ell\gamma}{2})}{\sinh(\frac{\ell\gamma_p}{2})}\right))\notag
\\
&-\mathrm{max}\left\{\mathrm{arcsin}\left(\frac{\sinh(\frac{\ell\gamma}{2})}{\sinh(\frac{\ell\gamma_p}{2})}\right),
\Theta(\delta,|n\ell\gamma-\tau|-\ell\gamma,\mathrm{arccosh}\left(\frac{\sinh(\frac{\ell\gamma}{2})}{\sinh(\frac{\ell\gamma_p}{2})}\right))\right\}.
\end{align}

Now for the case when $P$ is topologically a one-holed torus, the parameters $\tau$ and $\delta$ are similarly defined. The gap-angle is:
\begin{align}
&\mathrm{Gap}(P)=\mathrm{Gap}(\ell\gamma,\ell\gamma_p,\tau,\delta)\notag\\
&=2\mathrm{arcsin}\left(\frac{\cosh(\frac{\ell\gamma}{2})}{\cosh(\frac{\ell\gamma_p}{2})}\right)-\Theta(\delta,\ell\gamma\left\lceil\frac{\Psi-\tau}{\ell\gamma}\right\rceil+\tau,\mathrm{arccosh}\left(\frac{\sinh(\frac{\ell\gamma}{2})}{\sinh(\frac{\ell\gamma_p}{2})}\right))\notag \\
&-\Theta(\delta,\ell\gamma\left\lceil\frac{\Psi-(\ell\gamma-\tau)}{\ell\gamma}\right\rceil+\ell\gamma-\tau,\mathrm{arccosh}\left(\frac{\sinh(\frac{\ell\gamma}{2})}{\sinh(\frac{\ell\gamma_p}{2})}\right)),
\end{align}
where $\Psi$ is given by:
\begin{align*}
\Psi=\frac{1}{2}\log\left(\frac{\cosh^2(\delta)}{\sinh^2(\frac{\ell\gamma}{2})}-\frac{\cosh^2(\delta)}{\sinh^2(\frac{\ell\gamma_p}{2})}
\right).
\end{align*}

\subsection{Gap-angle calculations}

One trigonometric identity that we employ in this subsection that is not given in \cite{buser}, but may be derived from techniques outlined in chapter $2$ of \cite{buser}, relates to the $\Theta$ function:
\begin{align*}
\Theta(x,y,z)=\frac{1}{2}\mathrm{arccos}\left(\frac{2(\cosh(x)\cosh(y)\sinh(z)-\sinh(x)\cosh(z))^2}{(\cosh(x)\cosh(y)\cosh(z)-\sinh(x)\sinh(z))^2-1)}-1\right),
\end{align*}
where $\Theta$ measures one of the non-right angles as drawn in figure \ref{theta}.\newline
\newline

\begin{figure}[h]
\begin{center}
\includegraphics[scale=0.60]{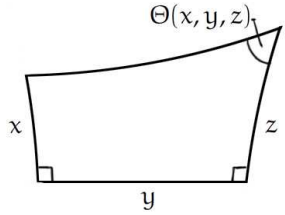}
\caption{Defining the $\Theta$ function.}
\label{theta}
\end{center}
\end{figure}

Since we have already covered the case when a lasso-induced pair of half-pants $P$ is embedded, we commence with the case where $P$ is a thrice-holed sphere with twisting number $n=0$. One way of thinking about a non-embedded pair of half-pants is to treat it as a pair of half-pants $P$ with a small triangular wedge $\triangle$ marked out on $P$ (as illustrated in figure \ref{0_3_universal}) indicating where the $P$ overlaps itself. A geodesic launched from $p\in P$ within the spiral-angle regions induces a lasso which induces $P$ unless it meets itself prematurely in $\triangle$. And any geodesic that meets $\triangle$ prior to self-intersecting normally must self-intersect prematurely since hyperbolic bigons don't exist. Therefore, the condition of meeting $\triangle$ prior to self-intersecting classifies all geodesics which need to be discounted from the spiral-angle to obtain the gap-angle. Let's consider this on the universal cover, as shown in figure \ref{0_3_universal}:

\begin{figure}[h]
\begin{center}
\includegraphics[scale=0.80]{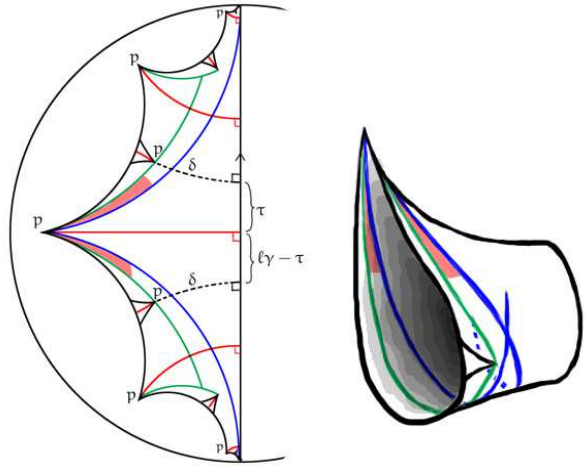}
\caption{The universal cover of a pair of lasso-induced half-pants homeomorphic to a thrice-holed sphere for $n=0$.}
\label{0_3_universal}
\end{center}
\end{figure}

Geodesics launched from a lift of $p$ which hit the nearest lifts of $\triangle$ result in geodesics which prematurely intersect and hence are excluded from the gap-angle of $P$. However, launched geodesics which meet other lifts of $\triangle$ without meeting the adjacent ones must intersect itself (hence forming a $P$-inducing lasso) prior to hitting $\triangle$, and so we see that our gap angle is given by the red regions in figure \ref{0_3_universal} and figure \ref{0_3_trig}.

\begin{figure}[h]
\begin{center}
\includegraphics[scale=0.60]{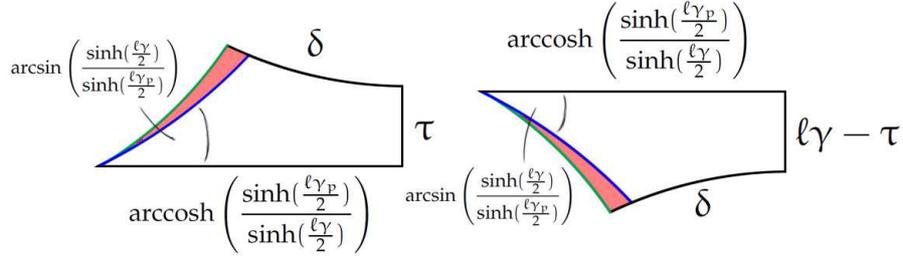}
\caption{Hyperbolic quadrilaterals from the universal cover of $P$.}
\label{0_3_trig}
\end{center}
\end{figure}

Therefore, the angle region is given by:
\begin{align*}
\Theta(\delta,\tau,\mathrm{arccosh}\left(\frac{\sinh(\frac{\ell\gamma}{2})}{\sinh(\frac{\ell\gamma_p}{2})}\right))+
\Theta(\delta,\ell\gamma-\tau,\mathrm{arccosh}\left(\frac{\sinh(\frac{\ell\gamma}{2})}{\sinh(\frac{\ell\gamma_p}{2})}\right))
\end{align*}
subtract the angle corresponding to launched geodesics which leave $P$ via its cuff before intersecting itself (not counting premature intersections). This gives us the $n=0$ term, modulo the introduction of maximum functions to account for the case when one of the spiral-regions is completely blocked off by $\triangle$.

\begin{figure}[h]
\begin{center}
\includegraphics[scale=0.40]{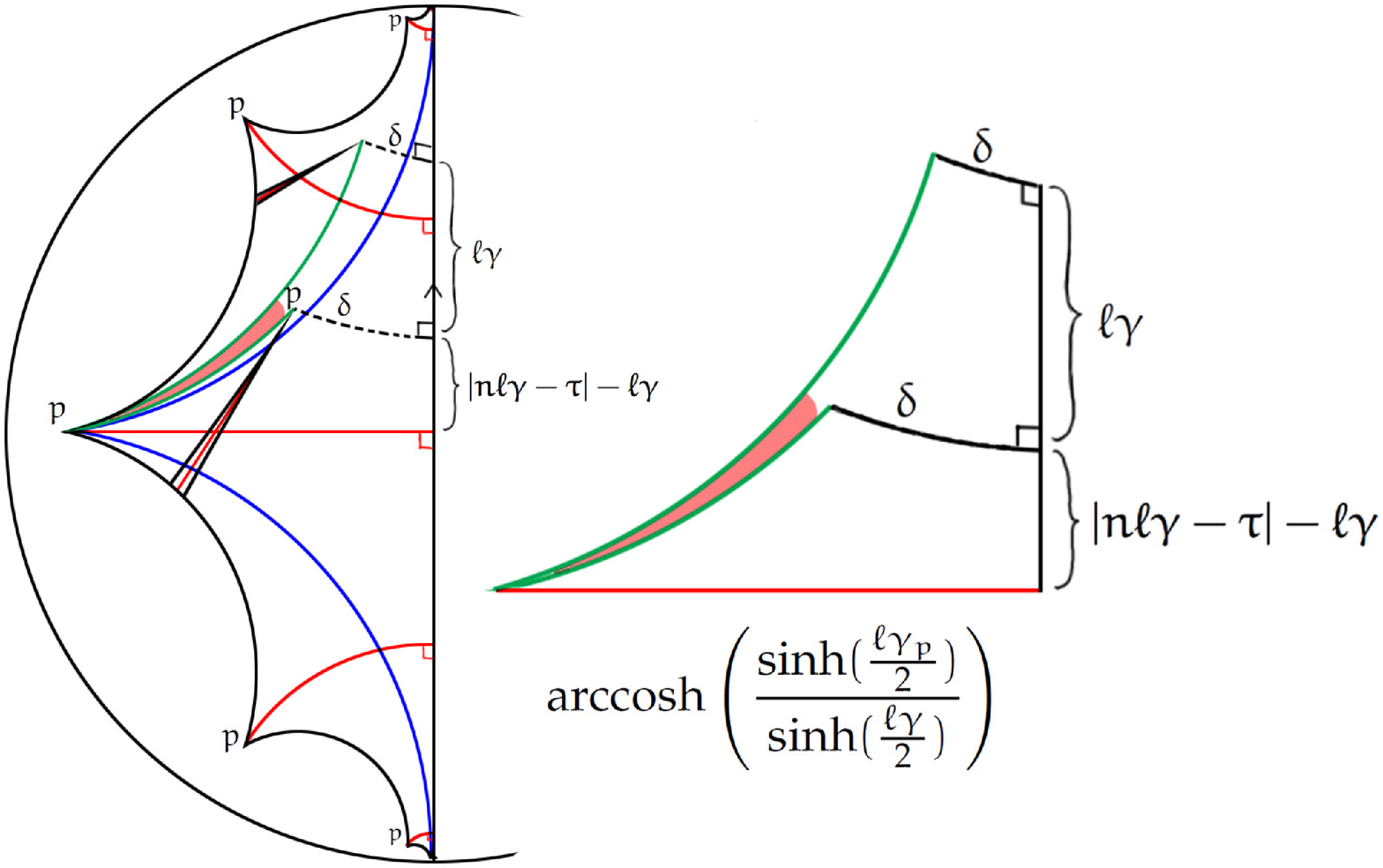}
\caption{The universal cover of a pair of lasso-induced half-pants homeomorphic to a thrice-holed sphere for $n\neq0$.}
\label{0_3_trig_n_not_0}
\end{center}
\end{figure}

For $n\neq 0$, we first note that one of our spiral-angle regions is completely blocked off by the overlapping triangle $\triangle$. When looked at on the universal cover of $P$, the geodesics which must again be discounted from the spiral angle regions are those that meet the nearest lifts of $\triangle$. Although we need to bear in mind that it is possible for an adjacent lift of $\triangle$ to lie outside the spiral-angle region. Combining these facts yields the $n\neq 0$ gap-angle, which we shade in red in figure \ref{0_3_trig_n_not_0}.
\newline
\newline
Finally, we consider the case when $P$ is topologically a one-holed torus. As before, a necessary condition for geodesics launched within the spiral regions to self-intersect prematurely is to enter $\triangle$. And as before, this is a sufficient condition because of the impossibility of bigons and because geodesics launched within the spiral region meet the cuff of $P$ precisely once prior to self-intersecting (they meet the cuff once due $\triangle$ intruding on $P$ via the cuff). Let us consider the universal cover in this case:

\begin{figure}[h]
\begin{center}
\includegraphics[scale=0.35]{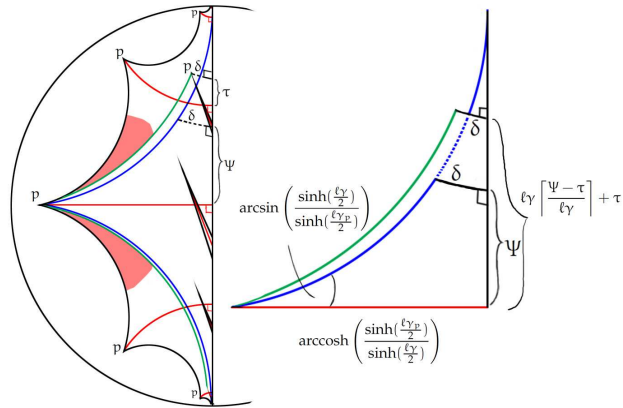}
\caption{The universal cover of a pair of lasso-induced half-pants homeomorphic to a one-holed torus.}
\label{1_1_trig}
\end{center}
\end{figure}

Since we may choose geodesics in the spiral regions which spiral arbitrarily closely to the cuff of $P$, there must be geodesics in the spiral regions which meet the triangular region. Within a single spiral region, as we vary the projection angle from launching almost parallel to the zipper of $P$ to the infinite simple geodesic that spirals around the cuff, there is a phase-shift geodesic (green) for each spiral-region that intersects $\triangle$ in such a way that all geodesics prior to it self-intersect normally, and all those that come after it self-intersect prematurely. This phase-shift geodesic hits the tip of $\triangle$, that is: it must hit $p$. The lifts of $\triangle$ that meet a chosen lift of one of the two phase-shift geodesics is also the first one that intersects the lift of the infinite simple spiraling geodesic (blue) bounding the relevant spiral-region. The figure on the right of figure \ref{1_1_trig} then enables us to calculate the desired gap-angle.

\section{Discussion}
Most of our analysis, including the trigonometry, is reasonably easily adapted to the context of hyperbolic surfaces with geodesic boundaries and small cone-angles. Specifically, a doubling construction may be used to extend our generalized Birman Series theorem to all bordered hyperbolic surfaces. Then, accounting for the fact that some of the angle measure shot from a small cone-angled point $p$ will now be taken up by geodesics which hit a boundary component, we can obtain a McShane identity with different summands depending on whether the relevant pair of half-pants has interior or exterior cuffs. In particular, we obtain the following porism: 

\begin{por}
Consider a finite-area hyperbolic surface $S$ with a single cone-point $p$ of angle $\theta_p\leq\pi$ and possibly with cusps, geodesic boundaries and other cone-points. Let $\mathcal{HP}_{\mathrm{int}}(S,p)$ and $\mathcal{HP}_{\mathrm{ext}}(S,p)$ respectively be the collection of embedded half-pants on $S$ with cuffs $\gamma$, as well as zipper $\zeta$ starting and ending at $p$. Then:
\begin{align}
\sum_{P\in\mathcal{HP}_{\mathrm{int}}(S,p)}\mathrm{arcsin}\left(\frac{\cosh(\frac{\ell\gamma}{2})}{\cosh(\frac{\ell\zeta}{2})}\right)-\mathrm{arcsin}\left(\frac{\sinh(\frac{\ell\gamma}{2})}{\sinh(\frac{\ell\zeta}{2})}\right)&\notag\\
+\sum_{P\in\mathcal{HP}_{\mathrm{ext}}(S,p)}\mathrm{arcsin}\left(\frac{\cosh(\frac{\ell\gamma}{2})}{\cosh(\frac{\ell\zeta}{2})}\right)&=\frac{\theta_p}{2},
\end{align}
where $\ell\gamma$ is the hyperbolic length of $\gamma$  if it's a closed geodesic, $0$ if $\gamma$ is a cusp and $i$ times the angle at $\gamma$ if it's a cone-point.
\end{por}

In particular, if $\theta_p\geq2\pi$, then equation \eqref{mainthm} holds true when we replace $\pi$ with $\frac{1}{2}\theta_p$. On the other hand, when $\theta\leq\pi$, our identity is equivalent to what is already known. The existence of pairs of pants in this scenario means that each of our half-pants in $\mathcal{HP}(S,p)$ is paired with precisely one other, such that they join to give an embedded pair of pants in $S$ with one boundary given by $p$, and the other two labeled as $\gamma_1$ and $\gamma_2$. The (possibly imaginary) lengths of $\gamma_1$ and $\gamma_2$, along with $\theta_p$, completely determine the geometry of this pair of pants. In particular, the length of the zipper $\zeta$ may be calculated using the following relation:
\begin{align}
\cosh^2\left(\frac{\ell\zeta}{2}\right)=\frac{\cosh^2(\frac{\ell\gamma_1}{2})+\cosh^2(\frac{\ell\gamma_2}{2})+2\cos(\frac{\theta_p}{2})\cosh(\frac{\ell\gamma_1}{2})\cosh(\frac{\ell\gamma_2}{2})}{\sin^2(\frac{\theta_p}{2})}\label{conversion}
\end{align}
Substituting this into our summands yields the main theorem of Tan-Wong-Zhang's generalization of McShane's identity to cone-surfaces \cite{tan_zhang_cone}. In fact, up to taking a limit or replacing certain geodesic lengths with complexified ones, many previously known McShane identities are an incarnation of the above porism. As an example, we derive Theorem $1.16$ of \cite{tan_zhang_cone} using algebraic manipulation.

\begin{proof}
Given a pair of half-pants $P_1$ with cuff and zipper respectively given by $\gamma_1$ and $\zeta$, the half-pants $P_1$ must be adjoined with a pair of half-pants $P_2$ with cuff and zipper respectively given by $\gamma_2$ and $\zeta$ as no cone-angles exceed ${\pi}$ by assumption. If $P$ is exterior, we take $\gamma_2$ to be the exterior cuff, then the summand associate to $P$ is:
\begin{align}
\mathrm{arcsin}\left(\frac{\cosh(\frac{\ell\gamma_1}{2})}{\cosh(\frac{\ell\zeta}{2})}\right)-\mathrm{arcsin}\left(\frac{\sinh(\frac{\ell\gamma_1}{2})}{\sinh(\frac{\ell\zeta}{2})}\right)+\mathrm{arcsin}\left(\frac{\cosh(\frac{\ell\gamma_2}{2})}{\cosh(\frac{\ell\zeta}{2})}\right).\label{ext_angle}
\end{align}
By converting $\mathrm{arcsin}$ into $\mathrm{arctan}$ and substituting in equation \eqref{conversion}, we obtain that:
\begin{align*}
\mathrm{arcsin}\left(\frac{\cosh(\frac{\ell\gamma_1}{2})}{\cosh(\frac{\ell\zeta}{2})}\right)&=\mathrm{arctan}\left(\frac{\sin(\frac{\theta_p}{2})\cosh(\frac{\ell\gamma_1}{2})}{\cos(\frac{\theta_p}{2})\cosh(\frac{\ell\gamma_1}{2})+\cosh(\frac{\ell\gamma_2}{2})}\right),\\
\mathrm{arcsin}\left(\frac{\sinh(\frac{\ell\gamma_1}{2})}{\sinh(\frac{\ell\zeta}{2})}\right)&=\mathrm{arctan}\left(\frac{\sin(\frac{\theta_p}{2})\sinh(\frac{\ell\gamma_1}{2})}{\cos(\frac{\theta_p}{2})\cosh(\frac{\ell\gamma_1}{2})+\cosh(\frac{\ell\gamma_2}{2})}\right).
\end{align*}
 Expressing $\mathrm{arctan}$ in terms of natural logarithms then gives us that:
\begin{align*}
 \mathrm{arcsin}\left(\frac{\cosh(\frac{\ell\gamma_1}{2})}{\cosh(\frac{\ell\zeta}{2})}\right)=\frac{1}{2i}\log\left(\frac{\exp(\frac{i\theta_p}{2})\cosh(\frac{\ell\gamma_1}{2})+\cosh(\frac{\ell\gamma_2}{2})}{\exp(\frac{-i\theta_p}{2})\cosh(\frac{\ell\gamma_1}{2})+\cosh(\frac{\ell\gamma_2}{2})}\right).
\end{align*}
Hence \eqref{ext_angle} becomes the summand in theorem $1.16$:
\[
\frac{\theta_p}{2}-\mathrm{arctan}\left(\frac{\sin(\frac{\theta_p}{2})\sinh(\frac{\ell\gamma_1}{2})}{\cos(\frac{\theta_p}{2})\cosh(\frac{\ell\gamma_1}{2})+\cosh(\frac{\ell\gamma_2}{2})}\right).
\]

 On the other hand, if $P$ is interior, then its associated summand is:
\begin{align}
\sum_{k=1,2}\mathrm{arcsin}\left(\frac{\cosh\frac{\ell\gamma_k}{2}}{\cosh\frac{\ell\zeta}{2}}\right)-\mathrm{arcsin}\left(\frac{\sinh(\frac{\ell\gamma_k}{2})}{\sinh(\frac{\ell\zeta}{2})}\right).\
\end{align}
Converting everything to logarithms and replacing $\ell\zeta$ using \eqref{conversion} as above yields:
\begin{align*}
\frac{1}{2i}&\log \exp({i\theta_p})\left(\frac{\cosh(\frac{\ell\gamma_1-i\theta_p}{2})+\cosh(\frac{\ell\gamma_2}{2})}{\cosh(\frac{\ell\gamma_1+i\theta_p}{2})+\cosh(\frac{\ell\gamma_2}{2})}\right)\left(\frac{\cosh(\frac{\ell\gamma_2-i\theta_p}{2})+\cosh(\frac{\ell\gamma_1}{2})}{\cosh(\frac{\ell\gamma_2+i\theta_p}{2})+\cosh(\frac{\ell\gamma_1}{2})}\right)\\
=\frac{1}{i}&\log \exp({\frac{i\theta_p}{2}})\left(\frac{\cosh(\frac{\ell\gamma_1-i\theta_p}{2})+\cosh(\frac{\ell\gamma_2}{2})}{\cosh(\frac{\ell\gamma_2+i\theta_p}{2})+\cosh(\frac{\ell\gamma_1}{2})}\right)\\
=\frac{1}{i}&\log\left(\frac{\exp(\frac{i\theta_p}{2})+\exp(\frac{\ell\gamma_1+\ell\gamma_2}{2})}{\exp(\frac{-i\theta_p}{2})+\exp(\frac{\ell\gamma_1+\ell\gamma_2}{2})}\right)\\
=2&\mathrm{arctan}\left(\frac{\sin(\frac{\theta_p}{2})}{\cos(\frac{\theta_p}{2})+\exp(\frac{\ell\gamma_1+\ell\gamma_2}{2})}\right),
\end{align*}
which is precisely the summand for interior pairs of pants for hyperbolic surfaces with small cone-angles.
\end{proof}


\bibliographystyle{abbrv}
\bibliography{bibtex.bib}

\end{document}